\documentclass{amsart}
\usepackage{amsmath,amsthm,amssymb}
\usepackage{graphicx}

\newtheorem{theorem}{Theorem}[section]
\newtheorem{proposition}[theorem]{Proposition}
\newtheorem{lemma}[theorem]{Lemma}

\newtheorem{conjecture}[theorem]{Conjecture}

\theoremstyle{definition}
\newtheorem{definition}[theorem]{Definition}
\newtheorem{remark}[theorem]{Remark}
\newtheorem{problem}[theorem]{Problem}

\newtheorem{notation}[theorem]{Notation}

\theoremstyle{theorem}

\title[Nonexistence of Stein structures on 4-manifolds]{Nonexistence of Stein structures on 4-manifolds and maximal Thurston-Bennequin numbers}
 
\author[Kouichi Yasui]{Kouichi Yasui}
\thanks{The author was partially supported by JSPS KAKENHI Grant Number 25800048.}
\date{September 9, 2015. Revised: November 3, 2015.}

\address{Department~of~Mathematics, Graduate School~of~Science, Hiroshima~University, 1-3-1 Kagamiyama, Higashi-Hiroshima, 739-8526, Japan}
\email{kyasui@hiroshima-u.ac.jp}
\begin{document}

\begin{abstract}
For a 4-manifold represented by a framed knot in $S^3$, it has been well known that the 4-manifold admits a Stein structure if the framing is less than the maximal Thurston-Bennequin number of the knot. In this paper, we prove either the converse of this fact is false or there exists a compact contractible oriented smooth 4-manifold (with Stein fillable boundary) admitting no Stein structure. Note that an exotic smooth structure on $S^4$ exists if and only if there exists a compact contractible oriented smooth 4-manifold with $S^3$ boundary admitting no Stein structure. 
\end{abstract}

\maketitle

\section{Introduction}\label{sec:intro}
Stein 4-manifolds have many applications to low-dimensional topology (cf.\ \cite{GS, OS1}), and thus characterizing a 4-manifold admitting a Stein structure is an important problem. 
This problem is related to the 4-dimensional smooth Poincar\'{e} conjecture. Indeed, every homotopy $S^4$ is diffeomorphic to $S^4$ if and only if every compact contractible oriented smooth 4-manifold with $S^3$ boundary admits a Stein structure (cf.\ Remark~\ref{rem:SPC4}). Due to this fact and usefulness of Stein corks (e.g.\ \cite{AY1}), the following generalization is well known to experts. 

\begin{problem}\label{intro:prob:contractible} Does every compact contractible oriented smooth 4-manifold admit a Stein structure?
\end{problem} 
The following case is particularly interesting, since $S^3$ is Stein fillable. 
Note that a closed oriented 3-manifold is called Stein fillable if it is diffeomorphic to the boundary of a compact Stein 4-manifold.

\begin{problem}\label{intro:prob:contra_fillable} Does every compact contractible oriented smooth 4-manifold with Stein fillable boundary admit a Stein structure?
\end{problem}

The main purpose of this paper is to relate this problem with another natural problem as a potential approach to this problem. We recall the latter problem. For a 4-manifold represented by a framed knot in $S^3$ (i.e.\ an \textit{oriented} 4-manifold obtained from $D^4$ by attaching a 2-handle along a framed knot, where the orientation is the one induced from $D^4$.), it has been well known that the 4-manifold admits a Stein structure if the framing is less than the maximal Thurston-Bennequin number of the knot (\cite{E2, G1}). However, the converse of this fact is an open problem.

\begin{problem}\label{intro:prob:tb}Assume that a framed knot in $S^3$ represents a 4-manifold admitting a Stein structure. Is the framing less than the maximal Thurston-Bennequin number of the knot?
\end{problem}

The affirmative answer to this problem clearly gives a simple characterization of a framed knot representing  a Stein 4-manifold. 
We remark that this problem is affirmative for many knots (e.g.\ positive torus knots) due to the adjunction inequality for Stein manifolds.  
In this paper, we prove the following. 
\begin{theorem}\label{thm:main}Either Problem~\ref{intro:prob:contra_fillable} or \ref{intro:prob:tb} has a negative answer. Moreover, if the answer to Problem~\ref{intro:prob:contra_fillable} $($resp.\ Problem~\ref{intro:prob:tb}$)$ is affirmative, then there exist infinitely many counterexamples to Problem~\ref{intro:prob:tb} $($resp.\ Problem~\ref{intro:prob:contra_fillable}$)$. 
\end{theorem}

In fact, we give infinitely many potential counterexamples to Problem~\ref{intro:prob:contra_fillable} (see Conjecture~\ref{conj:non-Stein}). This theorem easily follows from the following examples, which we construct from the potential counterexamples to Problem~\ref{intro:prob:contra_fillable}. 
\begin{theorem}\label{intro:thm:example}There exist infinitely many framed knots in $S^3$ satisfying the following conditions. 
\begin{itemize}
 \item Each framed knot represents a 4-manifold which is diffeomorphic to the boundary connected sum of a contractible 4-manifold with Stein fillable boundary 3-manifold and a compact Stein 4-manifold. 
 \item The framing of each framed knot is not less than the maximal Thurston-Bennequin number of the knot. Moreover, each framed knot can be chosen so that the framing is arbitrarily larger than the maximal Thurston-Bennequin number of the knot. 
\end{itemize}
\end{theorem}

To determine the maximal Thurston-Bennequin numbers of these knots, we give a cabling formula under a certain condition, utilizing a result of Rutherford \cite{Ru} on rulings. This formula might be of independent interest (Proposition~\ref{prop:cabling}). We remark that 
other cabling formulas under totally different conditions were given by Etnyre and Honda \cite{EH2} and Tosun \cite{T} using contact width, but it seems difficult to verify that a knot satisfies the conditions on contact width. Yet another cabling formula under a different condition, which can be easily checked, was given by Lidman and Sivek~\cite{LSi} using a technique of Etnyre and Honda \cite{EH}. 
 
\section{Construction of framed knots}

We first introduce contractible 4-manifolds. For integers $n$ and $k$, let $X_{n,k}$  be the smooth 4-manifold given by the handlebody diagram in Figure~\ref{fig:X_n,k}. Here the box $n$ denotes the $n$ right-handed full twists. Each $X_{n,k}$ is clearly contractible, and thus its boundary $\partial X_{n,k}$ is a homology 3-sphere. We remark that each $X_{0,k}$ is diffeomorphic to $D^4$, since the 2-handle goes over the 1-handle geometrically once after isotopy. Let $Y_{n,k}$ denote the boundary connected sum of $X_{n,k}$ and the total space of the $D^2$-bundle over $S^2$ with Euler number $-n$. Clearly $Y_{n,k}$ has the handlebody diagram shown in Figure~\ref{fig:Y_n,k}. 
\begin{figure}[h!]
\begin{center}
\includegraphics[width=1.2in]{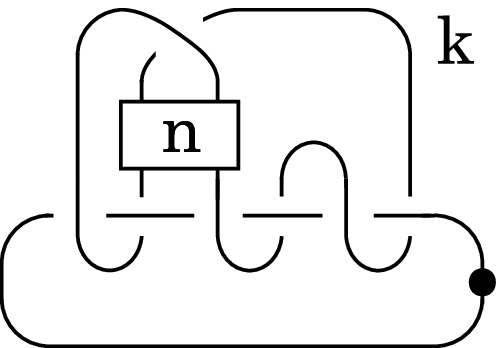}
\caption{$X_{n,k}$}
\label{fig:X_n,k}
\end{center}
\end{figure}
\begin{figure}[h!]
\begin{center}
\includegraphics[width=1.7in]{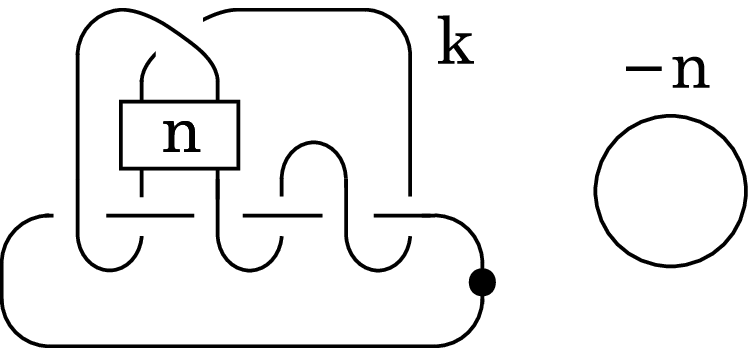}
\caption{$Y_{n,k}$}
\label{fig:Y_n,k}
\end{center}
\end{figure}

We here show that the boundary of $X_{n,k}$ is Stein fillable. We assume that the reader is familiar with handlebody descriptions of Stein 4-manifolds. For their details, we refer to \cite{GS}. 
\begin{lemma}\label{lem:Stein fillable}
$(1)$ $X_{n,k}$ admits a Stein structure for $n\geq 1$ and $k\leq 2n-1$.\smallskip \\
$(2)$ $\partial X_{n,k}$ is Stein fillable for $n\geq 2$ and $k\geq 4n-4$. 
\end{lemma}
\begin{proof}The claim (1) follows from the Stein handlebody diagram of $X_{n,k}$ in Figure~\ref{fig:Stein_X_n,k}. We prove (2). By blowing up $X_{n,k}$ for $n\geq 2$ and $k\geq 4n-4$, we obtain the first diagram in Figure~\ref{fig:fillable}. Modifying this 4-manifold preserving the boundary, we obtain the third diagram in the figure. Converting the 1-handle notation, we obtain the Stein handlebody given by the last diagram. We thus realized $\partial X_{n,k}$ as the boundary of a 4-manifold admitting a Stein structure. 
\end{proof}

\begin{figure}[h!]
\begin{center}
\includegraphics[width=2.1in]{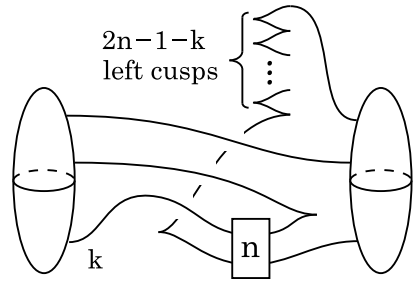}
\caption{Stein handle decomposition of $X_{n,k}$ $($$n\geq 1$ and $k\leq 2n-1$$)$}
\label{fig:Stein_X_n,k}
\end{center}
\end{figure}

\begin{figure}[h!]
\begin{center}
\includegraphics[width=4.8in]{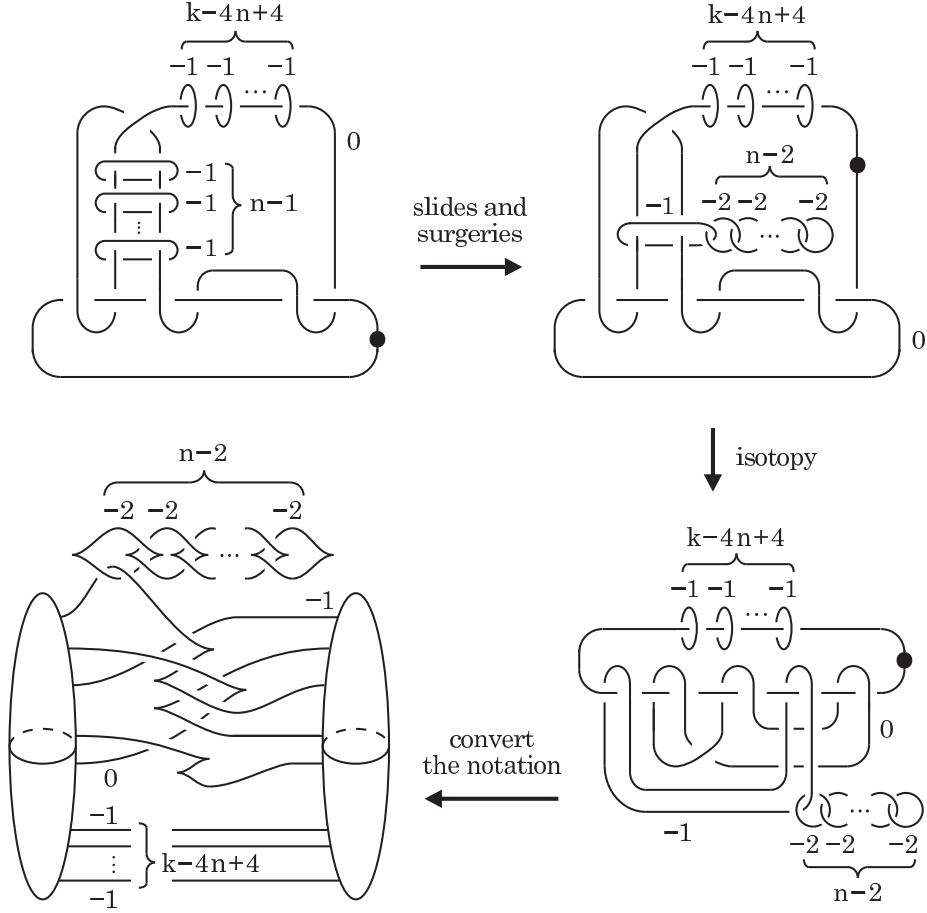}
\caption{4-manifolds bounded by $\partial X_{n,k}$ $($$n\geq 2$ and $k\geq 4n-4$$)$}
\label{fig:fillable}
\end{center}
\end{figure}

For an integral homology 3-sphere $M$, let $\lambda(M)$ denote the Casson invariant of $M$. For the basics of Casson invariants, we refer to \cite{Sa}. 
\begin{lemma}\label{lem:Casson}$\lambda(\partial X_{n,k})=-2n$ for any integers $n$ and $k$. Consequently, $\partial X_{n,k}$ is not homeomorphic to $S^3$ if $n\neq 0$. 
\end{lemma}
\begin{proof}A Rolfsen twist gives us the Dehn surgery diagram of $\partial X_{n,k}$ in Figure~\ref{fig:Casson}. Let $A_{n,k}$ be the $-\dfrac{1}{n}$-framed knot in the diagram. Ignoring $A_{n,k}$, this diagram gives $\partial{X_{0,k-4n}}$, which is diffeomorphic to $S^3$. Note that the $0$-framing of $A_{n,k}$ induced from this diagram of $\partial{X_{0,k-4n}}$ is equal to the $0$-framing of $A_{n,k}$ induced from a Seifert surface of $A_{n,k}$ in $S^3$. The surgery formula of the Casson invariant thus shows 
\begin{equation*}
\lambda(\partial X_{n,k})= \lambda(S^3)-\frac{n}{2}\Delta_{A_{n,k}}''(1)=-\frac{n}{2}\Delta_{A_{n,k}}''(1), 
\end{equation*}
where $\Delta_{A_{n,k}}$ denotes the Alexander polynomial of ${A_{n,k}}$. According to Corollary 1.7 in \cite{AKa}, we have $\lambda(\partial X_{1,k})=-{2}$. The above equality thus shows $\Delta_{A_{n,k}}''(1)=4$. Therefore the claim follows. 
\end{proof}
\begin{remark}This proof shows that $\partial X_{n,k}$ is obtained by a non-integral surgery along a knot in $S^3$. Due to \cite{GL}, this implies that $\partial X_{n,k}$ is an irreducible 3-manifold. 
\end{remark}

\begin{figure}[h!]
\begin{center}
\includegraphics[width=1.6in]{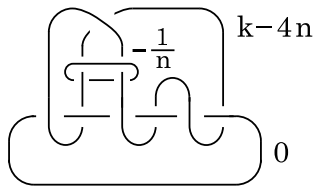}
\caption{$\partial X_{n,k}$}
\label{fig:Casson}
\end{center}
\end{figure}

Next we construct a framed knot by canceling the 1-handle of $Y_{n,k}$. Let us recall that, for a knot $K$ in $S^3$, the $(p,q)$-cable $C_{p,q}(K)$ of $K$ is defined to be a knot in $S^3$ which is a simple closed curve in the boundary $\partial \nu(K)$ of the tubular neighborhood $\nu(K)$ of $K$ representing the class $p[K']+q[\alpha]$ in $H_1(\partial \nu(K);\mathbb{Z})$. Here $\alpha$ is the positively oriented meridian of $K$, and $K'$ is the $0$-framing of $K$ induced from a Seifert surface of $K$.

For an integer $m$, let $K_m$ be the ribbon knot in $S^3$ given by Figure~\ref{fig:ribbon_knot}. For a positive integer $n$, we denote the $(n,-1)$ cable of $K_{m}$ by $K_{m,n}$. 
\begin{figure}[h!]
\begin{center}
\includegraphics[width=1.4in]{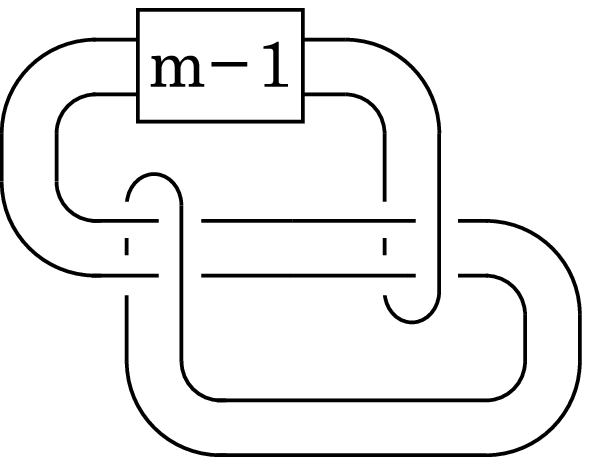}
\caption{$K_m$}
\label{fig:ribbon_knot}
\end{center}
\end{figure}

\begin{lemma}\label{lem:cancel}For integers $n,k$ with $n\geq 1$, $Y_{n,k}$ is diffeomorphic to the 4-manifold represented by the knot $K_{k-4n,n}$ with the framing $-n$. 
\end{lemma}
\begin{proof}We consider the diagram of $Y_{n,k}$ in Figure~\ref{fig:Y_n,k}. Sliding the $k$-framed knot over the $-n$-framed unknot twice, we obtain the upper diagram in Figure~\ref{fig:cancel}. Since the $(k-4n)$-framed knot goes over the 1-handle geometrically once after isotopy, the subhandlebody consisting of the 1-handle and this 2-handle is diffeomorphic to $D^4$. The intersection form of $Y_{n,k}$ tells that the $-n$-framed knot in the upper diagram becomes a $-n$-framed knot in $S^3$ after canceling the handle pair. By the definition of cable knots, one can easily check that this knot is the $(n,-1)$-cable of the unframed knot shown in the lower left diagram of Figure~\ref{fig:cancel}. Therefore it suffices to show that this unframed knot is isotopic to $K_{k-4n}$. By isotopy, we see that this knot is isotopic to the unframed knot in the lower right diagram. Canceling the handle pair, we easily see that the resulting knot is isotopic to $K_{k-4n}$. 
\end{proof}

\begin{figure}[h!]
\begin{center}
\includegraphics[width=4.0in]{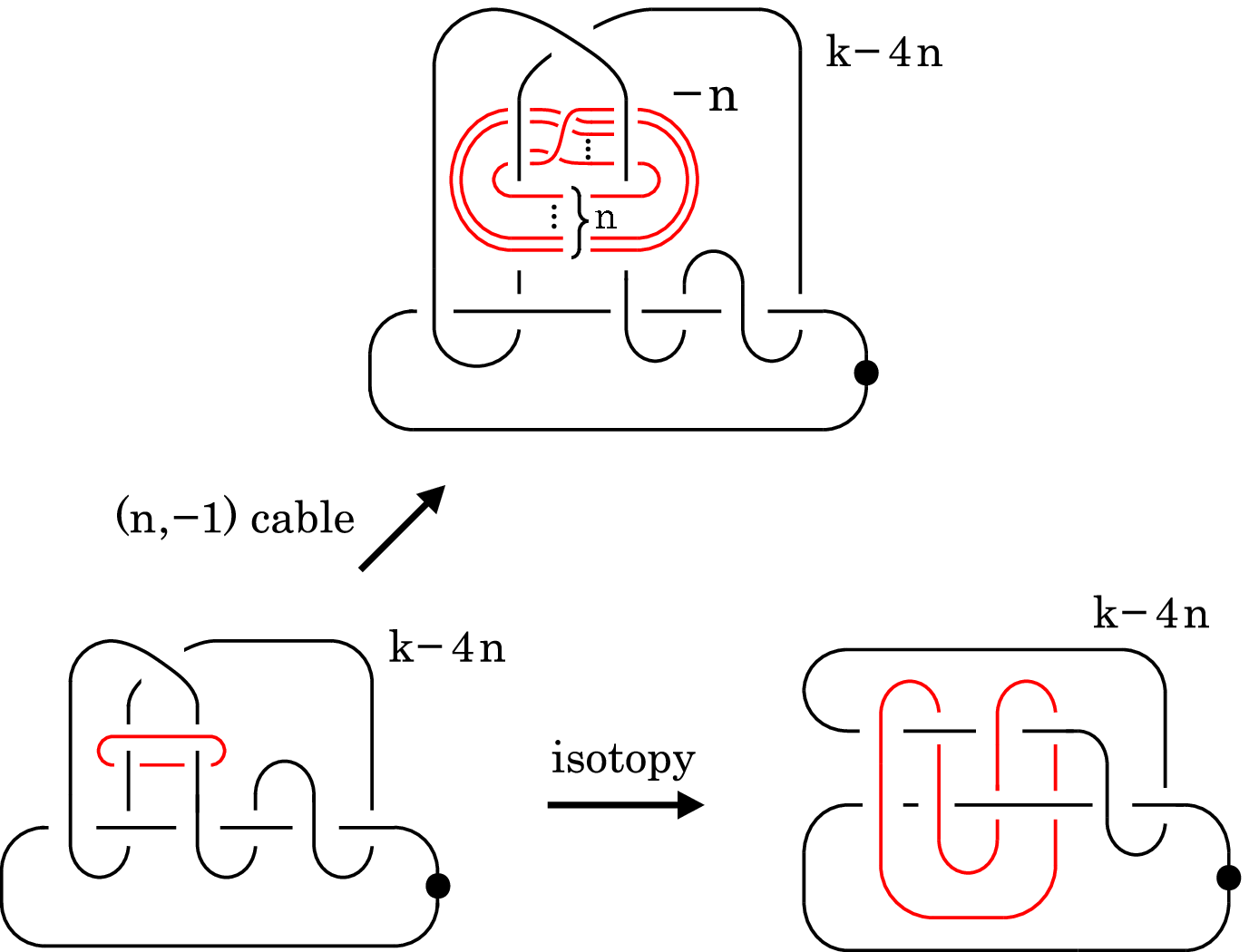}
\caption{}
\label{fig:cancel}
\end{center}
\end{figure}

\section{Rulings of Legendrian knots}In general, it is a difficult problem to determine the maximal Thurston-Bennequin number of a knot. In this section, we briefly recall rulings which provide a useful method for this problem. We refer to \cite{Ru} for details. For basics of contact topology and Legendrian knots, the readers can consult \cite{OS1}.

In this paper, a Legendrian knot in $S^3$ means the one with respect to the standard tight contact structure on $S^3$. We use the following notation. 
\begin{notation}For a Legendrian knot $\mathcal{K}$ in $S^3$, $tb(\mathcal{K})$ denotes the Thurston-Bennequin number of $\mathcal{K}$. A Legendrian representative of a (smooth) knot in $S^3$ is a Legendrian knot smoothly isotopic to the knot. For a (smooth) knot $K$ in $S^3$, the maximal Thurston-Bennequin number $\overline{tb}(K)$ of $K$ is the maximal value of $tb(\mathcal{K})$ of a Legendrian representative of ${K}$. 
\end{notation}

Let $\mathcal{K}$ be a Legendrian knot in $S^3$. By removing one point from $S^3$, we may assume that $\mathcal{K}$ is tangent to the standard tight contact structure $\textnormal{ker}(dz-ydx)$ on $\mathbb{R}^3=\{(x,y,z)\mid x,y,z\in \mathbb{R}\}$. We consider the front diagram (i.e.\ the projection into the $xz$ plane) of $\mathcal{K}$.  By planar isotopy, we may assume that all singular points of the front diagram $D$ of $\mathcal{K}$ have pairwise distinct $x$-coordinates. For a subset $\Lambda$ of the set of crossing points of the front diagram of $D$, let $D_\Lambda$ be the front diagram of a Legendrian link obtained from $D$ by \textit{smoothing} the neighborhood of each crossing point $\lambda\in \Lambda$ as shown in Figure~\ref{fig:smoothing}. We denote the $x$-coordinate of $\lambda\in \Lambda$ by $x_\lambda$. 

\begin{definition}$\Lambda$ is called a ruling of $\mathcal{K}$ if the following conditions hold. 
\begin{itemize}
 \item Each knot component $\mathcal{K}_i$ of the front diagram $D_\Lambda$ has exactly one left cusp and no self-crossing. We denote the upper and the lower horizontal strands by $U_i$ and $L_i$, respectively. 
 \item For each $\lambda\in \Lambda$, the two strands obtained by smoothing at the crossing point $\lambda$ belong to two distinct knot components of $D_{\Lambda}$. We denote the upper and the lower of these strands by $P_\lambda$ and $Q_\lambda$, respectively. 
 \item (Normality) For each $\lambda\in \Lambda$, there exist two knot components, denoted by $\mathcal{K}_i$ and $\mathcal{K}_j$, satisfying one of the following conditions (see also the lower part of Figure~\ref{fig:smoothing}). 
 \begin{enumerate}
 \item [(i)] $P_\lambda=L_i$ and $Q_\lambda=U_j$.
 \item [(ii)] $P_\lambda=L_i$ and $Q_\lambda=L_j$. Furthermore, the $z$-coordinate of $U_i$  is less than that of $U_j$ at $x=x_\lambda$. 
 \item [(iii)] $P_\lambda=U_i$ and $Q_\lambda=U_j$. Furthermore, the $z$-coordinate of $L_i$  is less than that of $L_j$ at $x=x_\lambda$. 
\end{enumerate}
\end{itemize}
\end{definition}

\begin{figure}[h!]
\begin{center}
\includegraphics[width=3.8in]{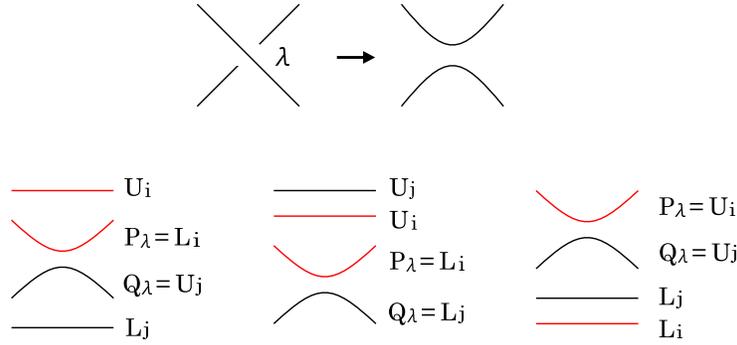}
\caption{Smoothing and the normality condition}
\label{fig:smoothing}
\end{center}
\end{figure}
We remark that there are many knots admitting Legendrian representatives with rulings (see \cite{Ka} and the references therein). However, there are also many knots (e.g.\ many negative torus knots) in $S^3$ which have no Legendrian representative admitting a ruling (cf.\ \cite{Ng1}). 

According to the following result of Rutherford, one can determine the maximal Thurston-Bennequin numbers of knots admitting Legendrian representatives with rulings. 
\begin{theorem}[Rutherford~\cite{Ru}]\label{thm:Rutherford}If $\mathcal{K}$ is a Legendrian knot in $S^3$ admitting a ruling, then $\overline{tb}(\mathcal{K})=tb(\mathcal{K})$. 
\end{theorem}
\begin{remark}Although finding a ruling of a Legendrian knot is not easy, the following result of Rutherford~\cite{Ru} characterizes the existence of a ruling: a Legendrian knot $\mathcal{K}$ admits a ruling if and only if the Kauffman bound of $\overline{tb}(\mathcal{K})$ is equal to  ${tb}(\mathcal{K})$. Here the Kauffman bound is the bound given by Kauffman polynomial. 
\end{remark}

\section{Determining the maximal Thurston-Bennequin numbers}
In this section, we determine $\overline{tb}(K_{m,n})$ for $m\geq 0$ and prove our main theorem. We first give a ruling for $K_m$.

\begin{lemma}\label{lem:ruling_K_m}For $m\geq 0$, the knot $K_m$ has a Legendrian representative with ${tb}=-2m-2$ which admits a ruling. Consequently, $\overline{tb}(K_m)=-2m-2$ for $m\geq 0$.
\end{lemma}
\begin{proof}It is not difficult to see that $K_m$ $(m\geq 0)$ has the Legendrian representative in Figure~\ref{fig:Legendrian_basic} (ignore dots). One can check that the set of dots in this diagram is a ruling. Thus Theorem~\ref{thm:Rutherford} shows that $\overline{tb}(K_m)$ is realized by this Legendrian representative. Calculating the Thurston-Bennequin number, we obtain the claim. 
\end{proof}

\begin{figure}[h!]
\begin{center}
\includegraphics[width=1.4in]{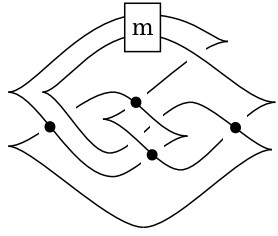}
\caption{A Legendrian representative of $K_m$ $(m\geq 0)$. The set of dots is a ruling. }
\label{fig:Legendrian_basic}
\end{center}
\end{figure}

Next we give a cabling formula of the maximal Thurston-Bennequin number of a knot, relying on the Rutherford's theorem. This might be of independent interest. 
\begin{proposition}\label{prop:cabling}Assume that a knot $K$ in $S^3$ admits a Legendrian representative with a ruling. Then for any relatively prime integers $p,q$ with $p\geq 2$ and $q\geq \overline{tb}({K})p+p$, the $(p,q)$-cable $C_{p,q}({K})$ of ${K}$ admits a Legendrian representative with a ruling and satisfies 
\begin{equation*}
\overline{tb}(C_{p,q}({K}))=\overline{tb}({K})p^2+(q-\overline{tb}({K})p)(p-1). 
\end{equation*}
\end{proposition}
\begin{proof}Let $\Lambda$ be a ruling of a front diagram $D$ of a Legendrian representative $\mathcal{K}$ of $K$. Let $D_p$ be the front diagram  consisting of $p$ copies of $D$ which are obtained by slightly moving $D$ to the $z$-direction. Figure~\ref{fig:copy_front} describes its local pictures near cusps and a crossing point of $D$. We denote the ruling $\Lambda$ in the $i$-th copy of $D$ by $\Lambda_i$ $(1\leq i\leq p)$. We fix one left cusp and one right cusp of $D$, and we modify $D_p$ near only these two cusps as shown in the upper side of Figure~\ref{fig:cusp_modification}. Here $r=q-tb(\mathcal{K})p-p$, and the box $\frac{r}{p}$ denotes the twists shown in the lower side of the figure. We denote the resulting front diagram by $D_{p,q}$, and let $\Gamma$ be the set of $(p-1)r$ crossing points in the  box $\frac{r}{p}$. One can see that $D_{p,q}$ gives a Legendrian representative $\mathcal{L}$ of the cable $C_{p,q}({K})$. 

Now let $\Phi$ be the subset of crossing points of $D_{p,q}$ defined by 
\begin{equation*}
\Phi=\Gamma\cup\Lambda_1\cup \Lambda_2\cup \dots \cup \Lambda_p. 
\end{equation*}
We show that $\Phi$ is a ruling of $D_{p,q}$. Since the local picture of $D_{p,q}$ around the fixed left cusps in Figure~\ref{fig:cusp_modification} does not contain any element of $\cup_{i=1}^p\Lambda_i$, we can easily check that each point of $\Gamma\subset D_{p,q}$ satisfies the condition of a ruling with respect to $\Phi$. We here consider the front diagram obtained from $D_{p,q}$ by smoothing all the crossing points in $\Gamma$. This diagram clearly has $p$ knot components each of which is isotopic to $D$. Therefore we see that $\Phi$ is a ruling of $D_{p,q}$, since each $\Lambda_i$ is a ruling of the $i$-th copy of $D$. Hence Theorem~\ref{thm:Rutherford} shows $\overline{tb}(C_{p,q}({K}))={tb}(\mathcal{L})$. Calculating ${tb}(\mathcal{L})$, we obtain the claim. 
\end{proof}
\begin{figure}[h!]
\begin{center}
\includegraphics[width=4.4in]{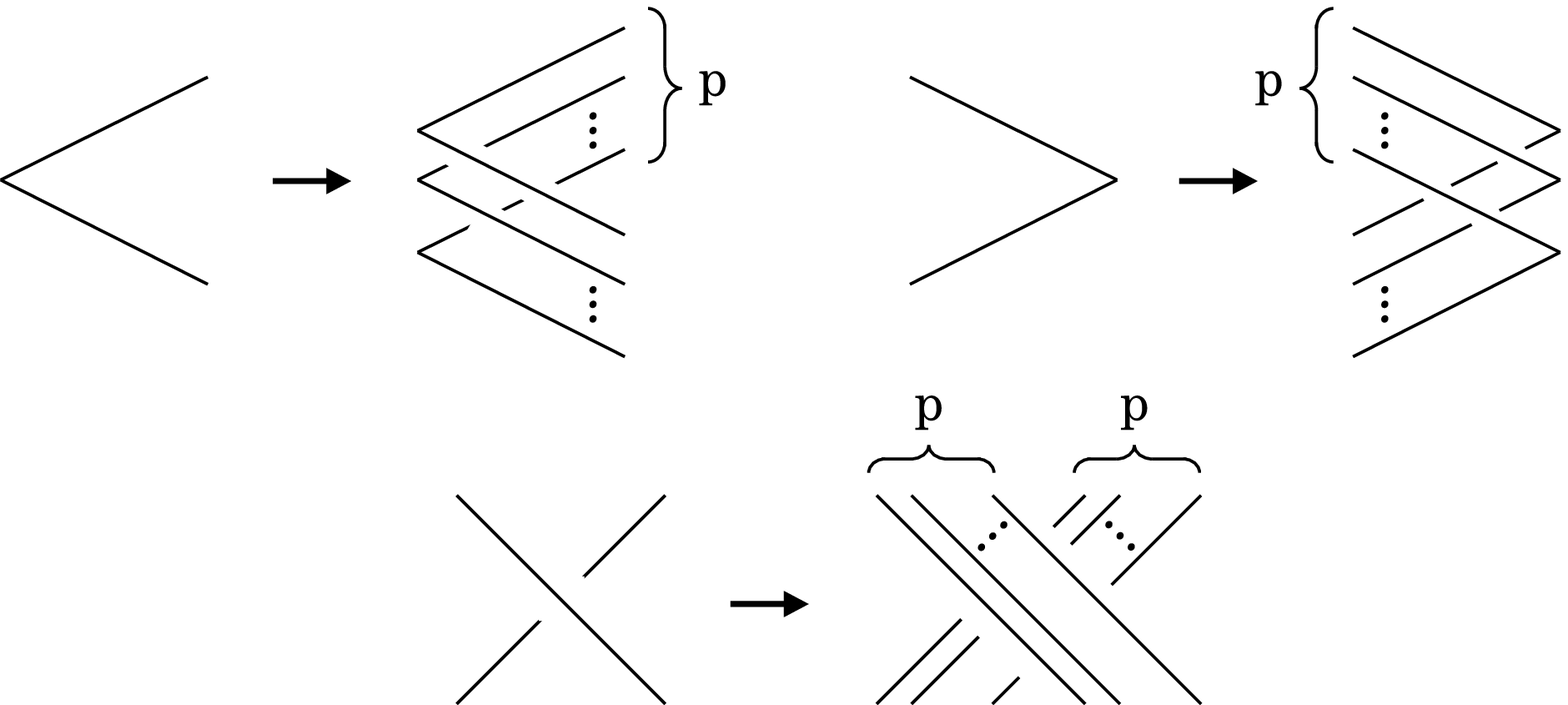}
\caption{}
\label{fig:copy_front}
\end{center}
\end{figure}
\begin{figure}[h!]
\begin{center}
\includegraphics[width=4.5in]{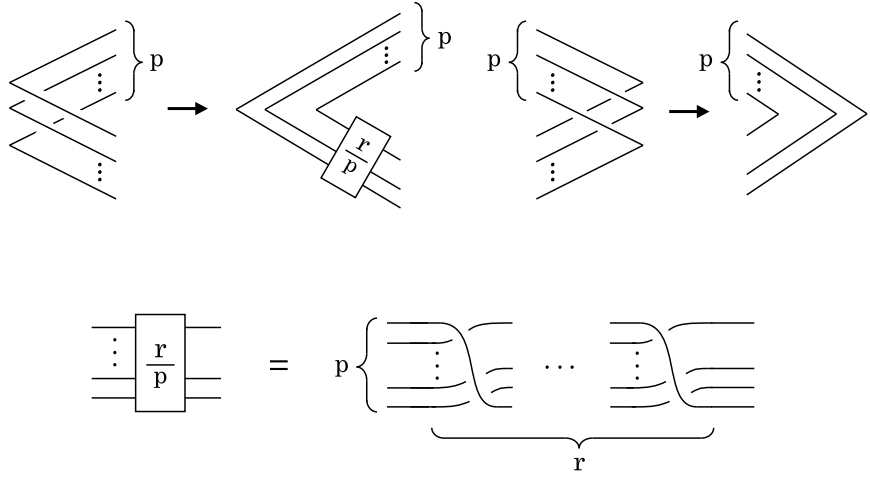}
\caption{Modification of fixed cusps ($r=q-tb(\mathcal{K})p-p\geq 0$)}
\label{fig:cusp_modification}
\end{center}
\end{figure}

\begin{remark}$(1)$ More generally, a similar formula must holds for satellite knots under an appropriate condition on patterns and companions. We do not pursue this point here. \smallskip\\
$(2)$ The same formula also holds if $C_{p,q}(\mathcal{K})$ is a link (i.e.\ $p,q$ are not relatively prime). This is because the aforementioned theorem of Rutherford also holds for links (\cite{Ru}). \smallskip\\
$(3)$ This formula does not hold for $q\leq \overline{tb}({K})p+p-1$. For example, if ${K}$ is unknot, then the formula does not hold for $q=\overline{tb}({K})p+p-1=-1$. 
\end{remark}
By this proposition and Lemma~\ref{lem:ruling_K_m}, we can clearly determine the value of $\overline{tb}(K_{m,n})$. 

\begin{proposition}\label{prop:tb(K_m,n)}For integers $m\geq 0$ and $n\geq 1$, 
\begin{equation*}
\overline{tb}(K_{m,n})=-2mn-3n+1.
\end{equation*}
\end{proposition}

\begin{remark}$(1)$ This proposition shows that $K_{m,n}$ is not isotopic to $K_{m',n}$ for $m > m'\geq 0$ and $n\geq 1$. \smallskip\\
$(2)$ In the $m<0$ case, the value of $\overline{tb}(K_{m,n})$ behaves very differently. Indeed, in \cite{Y9}, we prove $\overline{tb}(K_{m,n})=-1$ for $n\geq 2$ and $m\leq -4n+3$. (This result implies the existence of reducible Legendrian surgeries, disproving a conjecture in \cite{LSi}.) To show this, we construct a very complicated Legendrian representative, unlike the representative given by Proposition~\ref{prop:cabling}. 
\end{remark}

Now we are ready to prove our main result. 
\begin{proof}[Proof of Theorem~\ref{intro:thm:example}]Let $m,n$ be integers with $m\geq 0$ and $n\geq 2$. We consider the knot $K_{m,n}$ with $-n$-framing. 
By Lemma~\ref{lem:cancel}, the 4-manifold represented by this framed knot is diffeomorphic to $Y_{n,m+4n}$, which is the boundary connected sum of the contractible 4-manifold $X_{n,m+4n}$ and a compact Stein 4-manifold. Lemma~\ref{lem:Stein fillable} shows $\partial X_{n,m+4n}$ is Stein fillable, and Proposition~\ref{prop:tb(K_m,n)} shows $-n> \overline{tb}(K_{m,n})+m$. Therefore, the set 
\begin{equation*}
\{\text{$-n$-framed $K_{m,n}$}\mid m\geq 0,\; n\geq 2\}
\end{equation*}
 is an infinite family of framed knots satisfying the conditions of Theorem~\ref{intro:thm:example}. 
\end{proof}
\begin{proof}[Proof of Theorem~\ref{thm:main}]

We recall that the infinite family  
\begin{equation*}
\{\text{$-n$-framed $K_{m,n}$}\mid m\geq 0,\; n\geq 2\}
\end{equation*}
of framed knots satisfy the conditions of Theorem~\ref{intro:thm:example}, and that the 4-manifold represented by $-n$-framed $K_{m,n}$ is diffeomorphic to the boundary connected sum of the contractible 4-manifold $X_{n,m+4n}$ with Stein fillable boundary and a compact Stein 4-manifold. 

We first assume that Problem~\ref{intro:prob:contra_fillable} has an affirmative answer. 
Since the boundary connected sum of Stein handlebodies is a Stein handlebody (cf.\ \cite{GS}), the boundary connected sum of any compact Stein 4-manifolds admits a Stein structure. Therefore, the conditions of Theorem~\ref{intro:thm:example} and the assumption show that the above infinite family gives an infinite family of counterexamples to Problem~\ref{intro:prob:tb}.

We next assume that Problem~\ref{intro:prob:tb} has an affirmative answer. Then the boundary connected sum of $X_{n,m+4n}$ and a compact Stein 4-manifold admits no Stein structure for $n\geq 2$ and $m\geq 0$, since this 4-manifold is represented by $-n$-framed $K_{m,n}$, and its framing $-n$ is larger than $\overline{tb}(K_{m,n})$. This implies that the contractible 4-manifold $X_{n,m+4n}$ admits no Stein structure for $n\geq 2$ and $m\geq 0$. Therefore by Lemma~\ref{lem:Casson}, the set 
\begin{equation*}
\{X_{n,m+4n}\mid n\geq 2, \; m\geq 0\}
\end{equation*}
 gives an infinite family of counterexamples to Problem~\ref{intro:prob:contra_fillable}. 
\end{proof}

We propose potential counterexamples to Problem~\ref{intro:prob:contra_fillable}. 

\begin{conjecture}\label{conj:non-Stein}The compact contractible oriented smooth 4-manifold $X_{n,k}$ with Stein fillable boundary does not admit any Stein structure for $n\geq 2$ and $k\geq 4n$. 
\end{conjecture}
We remark that, if the answer to Problem~\ref{intro:prob:tb} is affirmative for $-n$-framed $K_{k-4n,n}$ ($n\geq 2$ and $k\geq 4n$), then $X_{n,k}$ does not admit any Stein structure (see the proof of Theorem~\ref{thm:main}). 
\begin{remark}Lemma~\ref{lem:Stein fillable} implies that $Y_{n,k}$ admits a Stein structure for $n\geq 2$ and $k\leq 2n-1$. Since Lemma~\ref{lem:cancel} tells that $Y_{n,k}$ is represented by $-n$-framed $K_{k-4n,n}$, it is natural to ask if the framing $-n$ is less than $\overline{tb}(K_{k-4n,n})$ for $n\geq 2$ and $k\leq 2n-1$. Contrary to the intuition coming from the standard cabling construction of a Legendrian representative, we  solve this question affirmatively in \cite{Y9}, giving a supporting evidence for Problem~\ref{intro:prob:tb}. 
\end{remark}

\begin{remark}\label{rem:SPC4}For the reader's convenience, we show that every homotopy $S^4$ is diffeomorphic to $S^4$ if and only if every compact contractible oriented smooth 4-manifold with $S^3$ boundary admits a Stein structure. This fact is well known to experts.

First assume the former statement is true, and let $C$ be a compact contractible oriented smooth 4-manifold whose boundary is diffeomorphic to $S^3$. Then the Freedman's theorem~\cite{Fr} tells that the 4-manifold $Z$ obtained from $C$ by attaching $D^4$ is homeomorphic to $S^4$. Hence $Z$ is diffeomorphic to $S^4$ by the assumption. Since an embedding of $D^4$ is unique, the complement $C\cong S^4-D^4$ is diffeomorphic to $D^4$. Thus $C$ admits a Stein structure. 

Next assume the latter statement is true, and let $Z$ be a smooth 4-manifold which is homotopic to $S^4$. Then the compact oriented 4-manifold $Z-D^4$ is simply connected and has the same homology group as that of one point. Hence, according to Whitehead's theorem, $Z-D^4$ is contractible. Thus $Z-D^4$ admits a Stein structure by the assumption. According to Eliashberg's theorems~\cite{E1, E3}, any Stein filling of $S^3$ is diffeomorphic to $D^4$. Hence $Z-D^4$ is diffeomorphic to $D^4$. Since attaching $D^4$ is unique, $Z=(Z-D^4)\cup D^4$ is diffeomorphic to $S^4$. 
\end{remark}
\section*{Acknowledgements}
The author would like to thank Selman Akbulut for fruitful conversations. This work was mainly done during the author's stay at Michigan State University from August 2014 to February 2015. He would like to thank them for their hospitality. He also thanks the referee for his/her helpful comments.

\end{document}